\documentclass[11pt]{article}
\usepackage[utf8]{inputenc}
\usepackage[T1]{fontenc}
\usepackage{amsmath}
\usepackage{amsthm}
\usepackage{amsfonts}
\usepackage{amssymb}
\usepackage{makeidx}
\usepackage{graphicx}
\usepackage{algorithm2e}
\usepackage{lmodern}
\usepackage{mathtools}
\usepackage{cite}
\usepackage{url}
\usepackage{xcolor}
\usepackage[left=2cm,right=2cm,top=1cm,bottom=2cm]{geometry}

\DeclareMathOperator{\seed}{seed}

\DeclareMathOperator{\EGE}{EGE}
\DeclareMathOperator{\gen}{gen}
\DeclareMathOperator{\Ann}{Ann}

\author{ P. Djagba   \thanks{NMU, Port Elizabeth. 
		\href{mailto:prudence.djagba@aims.ac.rw}{prudence@aims.ac.za}} 
	\and S. Juglal  \thanks{NMU,  Port Elizabeth. 
		\href{mailto:jan.hazla@aims.ac.rw}{Suresh.Juglal@mandela.ac.za }}}
\title{On Classical Prime Near-ring Modules}
\date{}

\usepackage{hyperref}
\usepackage{cleveref}

\newtheorem{thm}{Theorem}
\newtheorem{lem}[thm]{Lemma}

\newtheorem{cor}[thm]{Corollary}
\newtheorem{exa}[thm]{Example}
\newtheorem{rem}[thm]{Remark}

\theoremstyle{definition}
\newtheorem{defn}[thm]{Definition}

\theoremstyle{Proposition}
\newtheorem{pro}[thm]{Proposition}
\begin{document}
	\maketitle
	
	\begin{abstract} 
	In 2005, M. Behboodi introduced the notion of a classical prime ring module, which he showed is, in general, nonequivalent to a (Dauns) prime ring module. In this paper, we extended the idea of classical primeness to near-ring module. However, unlike in the ring case, we were able to define and distinguish between various types of classical prime modules. We investigate four of them here. We also prove some  properties about the annihilator. Finally we characterize the classical $m$-systems of $R$-ideals of near-ring modules.
  \end{abstract}

  \small
{\it 
 } \\
\normalsize
		\section{Introduction}

In 2005, Mahood Behboodi \cite{Behboodi} defined the notion of classical prime submodule of a ring module, M, as follows: A proper submodule, $P$ of $M$, is called classical prime if for all ideals, $A,B \subseteq R$ and for all submodules $N \subseteq M, \quad ABN \subseteq P $ implies that $AN \subseteq P$ or $BN \subseteq P.$ 

He went on to demonstrate that, in the case $M=R_R$ where $R$ is a commutative ring, classical prime submodules coincide with prime submodules but there may exist a left ideal $L$ in a noncommutative ring $R$ such that it is a classical prime submodule of $R$ but not a prime submodule of $R.$

Now suppose that $R$ is a near-ring and let $P$ be an $R$-ideal of a faithful $R$-module, $M$. Then we utilize the  above definition to define what is meant by $P$ being classical prime and hence define a classical prime module. However, due to the lack of one of the distributive properties as well as the fact that addition is, in general, non-commutative in a near-ring, we were able to define and distinguish between various types of classical primes module. We investigate four of them here, and show, by means of examples that these four types are nonequivalent within the class of all classical prime near-ring. Hence, in the near-ring case, we rename the notion of being classical prime as being prime  $0$-classical prime, and the other two as being 2-classical prime,  3-classical prime. It also turn out that for $v=0,2,3,c$, any $v$-prime $R$-ideal ($R$-module) is also $v$-classical prime but the reverse implication is, in general, not true.

In near-ring theory, $v$-prime ideals are closely linked to $m_v$-system sets. These sets are defined in (\cite{groenewald2010different}, Definition 1.30). We conclude this paper by defining the concepts of classical $m_v$-system ($v=0,2,3,c$) and show that an $R$-ideal $P$ of $M$ is $v$-classical prime if and only if its complement is a classical $m_v$-system.

Throughout this paper, $R$ will denote a zero-symmetric right near-ring (with identity only if specified) and M is faithful $R$-module.

The paper is organized as follows: Section 2 offers essential background definitions and results concerning near-rings, near-ring modules, and prime near-ring modules. These are necessary to introduce the concept of classical prime near-ring modules. In Section 3, we present the primary contribution of the paper, defining the notion of classical prime near-ring modules and providing several characterizations. Additionally, we demonstrate through examples the importance of this generalization from prime near-ring modules. Finally, in the last section, we conclude and propose potential avenues for future research.

		\section{Preliminaries}\label{sec:prelim}
		
		\subsection{Near-rings and near-ring modules}
		
		For basic concepts related to near-rings, we refer the readers to \cite{pilz2011near} for a more detailed presentation.

		\begin{defn}[Near-ring] %(\cite{meldrum1985near,pilz2011near})
			Let $(R, +, .)$ be a triple such that $(R, +)$ is a group, $(R, .)$ is a
			semigroup, and $(a+b).c = a.c+b.c$  for all $a, b, c \in R$. Then $(R, +, .)$ is
			called a \emph{(right) near-ring}.
		\end{defn}
		
		Let $(R, +, .)$  be a near-ring. By distributivity it is easy to see that for all $r \in R$ we have $ 0 . r = 0$. 
		However,
		it is not true in general that $r . 0  = 0$ for all $r \in R$. Define $R_0 = \{r  \in R : r . 0= 0\}$
		to be the zero-symmetric part of $R$. A near-ring is called \emph{zero-symmetric} if $R = R_0$ i.e.,
		$0 . r = r . 0 = 0$ for all $r \in R$. We will denote $R^*= R \backslash \{0 \}.$

		\begin{defn}[Near-field]
			Let $(R,+,.)$ be a  near-ring.
			If in addition $(R^*, .)$ is a group,
			then $(R, +, .)$ is called a 
			\emph{(right) near-field}.
			
			We will call a near-field $R$ \emph{proper} if
			it is not a skew-field, that is if there
			exist $a,b,c\in R$ such that  $a. (b +c)\ne a.b+a.c$.
		\end{defn}
	It is known
		that the additive group of a near-field is abelian \cite{pilz2011near}.
		The assumption that $R$ is zero-symmetric is only needed to exclude a degenerate
		case of $(\mathbb{Z}_2,+)$ with multiplication defined as $a.b=b$
		(see Proposition~8.1 in~\cite{pilz2011near}).

		\begin{defn}[$R$-Module] %(\cite{beidleman1966near})
			A triple $(M, +, \circ) $ is called a \emph{(left) near-ring module} over a (right)
			near-ring $R$ if $(M,+)$ is a group and $\circ: R \times M \to M$ such that $(r_1 +r_2) \circ m=  r_1 \circ m +r_2 \circ m$ and $(r_1.r_2) \circ m =  r_1. (r_2 \circ m)$ for all $r_1, r_2 \in R$ and $m \in M$.
			
			We write $M_R$ to denote that $M$ is a (left) near-ring module over a (right) near-ring $R$.
		\end{defn}

		As is usual, from now on we will use $\cdot$ or  simply concatenation for both near-ring multiplication
		and vector-scalar multiplication. 
		We also define the $R$-module $R^n$ (for some fixed $n\in \mathbb{N}$) with element-wise addition and element-wise scalar multiplication $R \times R^n \to R^n$ given by $ r \cdot (v_1,v_2,\ldots,v_n) = (r v_1,rv_2,\ldots,rv_n)$. 
			\begin{defn}[$R$-subgroup] %(\cite{beidleman1966near})
        \label{def:subgroup} Let $R$ be a near-ring. 
			A subset $H$ of $R$ is called a (two sided or invariant)  \emph{$R$-subgroup} of $R$ if:
   \begin{itemize}
       \item[(a)]  $H$ is a subgroup of $(R, +)$,
       \item[(b)]   $RH  \subseteq H$,
       \item[(c)]   $HR  \subseteq H$.
   \end{itemize}

		\end{defn}
   If in the above definition, (a) and (b) are satisfied , the $H$ is called a left $R$-subgroup whereas if  (a) and (c) are satisfied, then $H$ is called a right $R$-subgroup. If $H$ is a subgroup of $R,$ this will be denoted by $H \leq R.$
  \begin{rem}
      An $R$-subgroup $H$ of $R$ is called a normal subgroup if for all $r \in R$ and for all $h \in H$, we have $r+h-r \in H.$
  \end{rem}
  		\begin{defn}[Ideal] %(\cite{beidleman1966near})
        \label{def:subgroup} Let $R$ be a near-ring. 
			A subset $I$ of $R$ is called a (two sided )  ideal of $R$ if:
   \begin{itemize}
       \item[(a)]  $I$ is a normal subgroup of $(R, +)$,
       \item[(b)]   $IR  \subseteq I$,
       \item[(c)]   $r_1(r_2+i)-r_1r_2 \in I$ for all $r_1.r_2 \in R$ and $\in I$.
   \end{itemize}
		\end{defn}

     If in the above definition, (a) and (b) are satisfied , the $I$ is called a righrt ideal of $R$ whereas if  (a) and (c) are satisfied, then $I$ is called left ideal of $R$. If $H$ is a subgroup of $R,$ this will be denoted by $H \leq R.$ An ideal I of $R$ will be denoted by $I \lhd R$.
		\begin{defn}[$R$-submodule] %(\cite{beidleman1966near})
        \label{def:subgroup}
			A subset $H$ of a near-ring module $M_R$ is called an $R$-submodule of $M$ if $H$
			is a subgroup of $(M, +)$ and $RH = \{rh: h \in H, r  \in R\} \subseteq H.$ ( We denite this by $H \leq_R M$
		\end{defn}
		\begin{defn}[$R$-ideal] %(\cite{beidleman1966near})
			Let $M_R$ be a near-ring module. $N$ is an $R$-ideal of $M_R$ if $(N, +)$ is a normal subgroup of $(M, +)$, and $r(m + n) - rm  \in N$ for all $m \in M$, $n \in N$ and $r \in R$. (We will denote this by  $P \lhd_R M).$
		\end{defn}

  		\begin{defn}[] %(\cite{beidleman1966near})
        \label{def:subgroup} An $R$-module $M$ is called:
   \begin{itemize}
       \item[(a)]  monogenic if there exists an $m \in M$ (called generator of $M$) such that $Rm=M.$ 
       \item[(b)]  faithful if $r \in R$ and $rM=0$ implies $r=0.$
   \end{itemize}
		\end{defn}
		
Let $M$ be near-ring module. Then note that every $R$-ideal  of $M$  is an $R$-submodule    but the converse is not rue in general. In the case of a ring module the concept of $R$-ideal and   $R$-submodule coincide. 

\begin{defn}[ \cite{Juglal2008}] Let $P $ be a proper two-sided ideal of the near-ring $R$.  Then $P$ is called:
      
   \begin{itemize}
       \item[(a)]  $0$-prime if for all two-sided ideals $A,B$ of $R,$ $AB \subseteq P$ implies $A  \subseteq P$ or $B  \subseteq P.$
        \item[(b)]  $1$-prime if for all left ideals $A,B$ of $R,$ $AB \subseteq P$ implies $A  \subseteq P$ or $B  \subseteq P.$
        \item[(c)]  $2$-prime if for all $R$-subgroups $A,B$ of $R,$ $AB \subseteq P$ implies $A  \subseteq P$ or $B  \subseteq P.$
              \item[(d)]  $3$-prime if   $(aRb) \subseteq P$ implies $a \in P$ or $b \in P$ for all $a,b \in R.$
              \item[(e)]  $c$-prime if   $ab \in P$ implies $a \in P$ or $b \in P$ for all $a,b \in R$. (Here $c$-prime means completely prime).

   \end{itemize}
		\end{defn}
		The near-ring $R$ is said to be $i$-prime for $i \in  \{0,1,2,3,c \}$ if the zero ideal $ \{0 \}$ is $i$-prime. 
		\subsection{Near-vector spaces of the form $R^n$}
		\begin{defn}[Near-vector space] %(\cite{beidleman1966near})
			Let $M_R$ be a near-ring module  and $R$ a near-field. $M_R$ is called a 
			\emph{(Beidleman) near-vector space} if $M_R$ is a near-ring module which is a direct sum of  submodules that have no proper $R$-subgroups.
			We say that a near-vector space is finite
			dimensional if it is such a finite direct sum.
		\end{defn}

		\begin{defn}[$R$-module isomorphism]
			Let $M_R$ and $N_R$ be two modules. A function $\Phi:M\to N$ is a $R$-module isomorphism
			if it is a bijection that respects $\Phi(m+n)=\Phi(m)+\Phi(n)$ and $\Phi(mr)=\Phi(m)r$
			for every $m,n\in M$ and $r\in R$.
		\end{defn}
		
		\begin{thm}[\cite{djagbahowell18,beidleman1966near}] Let $R$ be a (right) near-field and $M_R$ a (left) near-ring module. $M_R$ is
			a finite dimensional near-vector space if and only if $ M_R $ is isomorphic to $R^n$ for some positive integer 
			$n$.
		\end{thm}

		%\subsection{EGE algorithm}
		In \cite{djagbahowell18, djagbathesis,djagbaprins, djagba}, $R$-submodules of finite dimensional
		near vector spaces have been classified using
		the Expanded Gaussian  Elimination ($\EGE$) algorithm.
		This algorithm is used to construct the smallest  $R$-submodules containing
		given finite set of vectors.
		(Such  $R$-submodules exist since any intersection of  $R$-submodules is 
		a  $R$-submodule.)
		
		\begin{defn}%(\cite{djagbahowell18})
			Let $V$ be a set of vectors. Define $\gen(V)$ to be the intersection of all $R$-submodules containing $V$.
		\end{defn}

		Let $M_R$ be a near-ring module. Let $V\subseteq M_R$ and let $T$ be an $R$-subgroup of $M_R$.
		\begin{defn}[Seed set]%(\cite{djagba})
			We say that \emph{$V$ generates $T$} if $\gen(V)=T$.
			In that case we say that $V$
			is a \emph{seed set} of $T$. 
			We also define the \emph{seed number} $\seed(T)$
			to be the cardinality of a smallest seed set of $T$.
		\end{defn}
		
		In \cite{djagbahowell18} it was proved that each $R$-subgroup is
		a direct sum of modules $u_iR$ of a special kind:
		
		\begin{thm}[Theorem~5.12 in~\cite{djagbahowell18}]
			\label{thm:ege}
			Let $R$ be a proper near-field
			and $\{v_1,\ldots,v_k\}$ be vectors in $R^n$. Then,
			$\gen(v_1,\ldots,v_k)=\bigoplus_{i=1}^\ell Ru_i$,
			where the $u_i$ are rows of some matrix $U=(u_{ij})\in R^{\ell\times n}$
			such that each of its columns has at most one non-zero entry.
		\end{thm}
		
		In particular, basis vectors $u_1,\ldots,u_\ell$ from
		\Cref{thm:ege} have mutually disjoint supports.
		\Cref{thm:ege} was proved by analyzing an explicit procedure
		termed ``Expanded Gaussian Elimination (EGE)''. This procedure
		takes $v_1,\ldots,v_k$ and outputs $u_1,\ldots,u_\ell$.
		Later on we will make use of the following corollary:

		\begin{thm}[Corollary~15 in~\cite{djagbajan}]
  \label{thm:ege1}
			Let $R$ be a proper near-field and $T\subseteq R^n$. Then,
			$T$ is an  $R$-submodule if and only if $T=\bigoplus_{i=1}^\ell Ru_i$
			for some nonzero vectors $u_1,\ldots,u_\ell$ with mutually disjoint supports.
		\end{thm}

\begin{thm}[Corollary~6.3 in~\cite{djagbahowell18}]
			\label{thm:ege}
The subspaces (or $R$-ideals) of $R^n$ are all of the form $ S_1 \times S_2 \times \cdots \times S_n$ where $S_i= \{ 0 \}$ or $S_i=R$ for $i=1, \ldots,n$.
\label{th4}
\end{thm}
  \subsection{Prime near-ring modules}
Let's denote  $\widetilde{P}=(P:M)_R= \{ r \in R: rM \subseteq P\}$.
  \begin{lem} (\cite{Juglal2008})
      If $P \lhd_R M$, then $\widetilde{P}$ is an ideal of $R.$
  \end{lem}
		  
  		\begin{defn}[ \cite{Juglal2008}] Let $P \lhd R $.  Then $P$ is called:
      
   \begin{itemize}
       \item[(a)]  $0$-prime if for all ideals $A,B$ of $R,$ $AB \subseteq P$ implies $A \subseteq P$ or $B\subseteq P.$
       \item[(b)]  $2$-prime if for all left $R$-subgroups $A,B$ of $R,$ $AB \subseteq P$ implies $A \subseteq P$ or $B\subseteq P.$
              \item[(c)]  $3$-prime if for all $a,b \in R,$ and $aRb \subseteq P$ implies $a \in P$ or $b \in P.$ 
\item[(d)]  (completely) $c$-prime if for all $a,b \in R,$ and $aRb \subseteq P$ implies $a \in P$ or $b \in P.$ 
   \end{itemize}
		\end{defn}

Dauns defined $v$-prime as follows:
  		  
  		\begin{defn}[Dauns \cite{Dauns1978}]
    
    \label{dauns} Let $P \lhd_R M $ such that $RM \nsubseteqq P$.  Then $P$ is called:
      
   \begin{itemize}
       \item[(a)]  $0$-prime (Dauns prime) if  $AB \subseteq P$ implies $A M \subseteq P$ or $B\subseteq P$ for all ideals $A$ of $R$ and for all $R$-submodules, $B$ of $M.$ 
       \item[(b)] $2$-prime if  $AB \subseteq P$ implies $A M \subseteq P$ or $B\subseteq P$ for all left $R$-subgroups $A$ of $R$ and for all $R$-submodules, $B$ of $M.$ 
              \item[(c)]  $3$-prime if  $aRm \subseteq P$ implies $a M \subseteq P$ or $m \in P$ for all $a \in R$ and $m \in M.$
 \item[(d)] (completely) $c$-prime  if  $rm \in P$ implies $r M \subseteq P$ or $m \in P$ for all $r \in R$ and $m \in M.$
   \end{itemize}
		\end{defn}
  In \cite{juglal2008prime} S. Juglal defined $v$-prime as follows
  		\begin{defn}[ \cite{Juglal2008}] 
    \label{suresh}
    Let $P \lhd_R M $ such that $RM \nsubseteqq P$.  Then $P$ is called:
      
   \begin{itemize}
       \item[(a)]  $0$-prime  if  $AB \subseteq P$ implies $A M \subseteq P$ or $B\subseteq P$ for all ideals $A$ of $R$ and for all $R$-ideals, $B$ of $M.$ 
       \item[(b)] $2$-prime  if  $AB \subseteq P$ implies $A M \subseteq P$ or $B\subseteq P$ for all left $R$-subgroups $A$ of $R$ and for all $R$-submodules, $B$ of $M.$ 
              \item[(c)]  $3$-prime  if  $aRm \subseteq P$ implies $a M \subseteq P$ or $m \in P$ for all $a \in R$ and $m \in M.$
 \item[(d)] (completely)  $c$-prime  if  $rm \in P$ implies $r M \subseteq P$ or $m \in P$ for all $r \in R$ and $m \in M.$

   \end{itemize}
		\end{defn}

  Note that Definition \ref{suresh} implies Definition \ref{dauns}. But in general, Daun's definition does not necessarily imply S. Juglal's definition since every $R$-submodule of $M$ is not always an $R$-ideal. In general, it is true that every $R$-ideal of $M$ is an $R$-submodule. 
  \begin{lem}
      Let $P \lhd_R M$ such that $RM \nsubseteq P$ and let $v=0,2,3,c.$ If $P$ is a $v$-prime $R$-ideal of $M,$ then $\widetilde{P}$ is a $v$-prime ideal of $R.$
  \end{lem}
  \begin{proof}~
  \begin{itemize}
    \item $v=0:$ Let $A,B$ be ideals of $R$ such that $AB \subseteq \widetilde{P}.$ Then $ABM \subseteq P,$ so that for all $m \in M, A(Bm) \subseteq P.$ Since $P$ is a $0$-prime $R$-ideal and $Bm$ is an $R$-submodule of $M,$ we have that $AM \subseteq P$ or $Bm \subseteq P$  for all $m \in M.$ If $AM \subseteq P,$ then $A \subseteq (P:M)=\widetilde{P}$ and we are done. If $Bm \subseteq P$ for all $m \in M,$ then $BM \subseteq P$ and hence $B \subseteq (P:M)=\widetilde{P}.$ Therefore, $\widetilde{P}$ is a $0$-prime ideal of $R.$
      \item $v=2:$ Similar to the previous case.
      \item  $v=3:$ Let $x,y \in R$ such that $xRy \subseteq \widetilde{P}.$ Suppose that $y \notin \widetilde{P}.$ Then $yM \nsubseteq P$ implies that there exists an $m \in M$ such that $ym \notin P.$ Now $xRy \subseteq \widetilde{P}$ implies that $xRym \subseteq P.$ Since $P$ is a $3$-prime $R$-ideal and $ym \notin P,$ we must have that $xM \subseteq P$ ie. $x \in \widetilde{P}.$ Thus $\widetilde{P}$ is a $3$-prime ideal of $R.$
  \end{itemize}

  \end{proof}
		\section{Classical prime near-ring module}

To extend the idea of classical primeness to near-ring module,  we now introduce the notion of classical prime ideals as follows

 		\begin{defn}[] Let $P \lhd R $.  Then $P$ is called:
      
   \begin{itemize}
       \item[(a)]  $0$-classical prime  if for all ideals $A,B$ of $R,$ and for all ideals $I$ of $R,$ $ABI \subseteq P$ implies $A I \subseteq P$ or $B I \subseteq P.$
       \item[(b)]  $2$-classical prime  if for all left $R$-subgroups $A,B$ of $R,$ and for all ideals $I$ of $R,$ $ABI \subseteq P$ implies $A I \subseteq P$ or $B I \subseteq P.$
              \item[(c)]  $3$- classical prime  if   $(aR)(bR)I \subseteq P$ implies $aI \in P$ or $b I\in P$ for all $a,b \in R,$ and
and for all ideals $I$ of $R.$
         \item[(d)] (completely)  $c$-classical prime  if   $(aRb)I \subseteq P$ implies $aI \in P$ or $b I\in P$ for all $a,b \in R,$ and
and for all ideals $I$ of $R.$
   \end{itemize}
		\end{defn}
We also introduce the concept of classical prime near-ring module as follow

  		\begin{defn}[]
    \label{new}
    Let $P \lhd_R M $ such that $RM \nsubseteq P$.  Then $P$ is called:
      
   \begin{itemize}
       \item[(a)]  $0$-classical prime  if  $ABN \subseteq P$ implies $A N \subseteq P$ or $BN \subseteq P$ for all ideals $A, B$ of $R$ and for all $R$-submodules, $N$ of $M.$ 
       \item[(b)] $2$-classical prime  if  $ABN \subseteq P$ implies $A N \subseteq P$ or $BN \subseteq P$ for all left $R$-subgroups $A, B$ of $R$ and for all $R$-submodules, $N$ of $M.$ 
              \item[(c)]  $3$-classical prime  if  $(aR)(bR)N \subseteq P$ implies $a N \subseteq P$ or $bN \in P$ for all $a,b \in R$ and for all $R$-submodules, $N$ of $M.$ 

                \item[(c)] (completely)  $c$-classical prime  if  $(aRb)N \subseteq P$ implies $a N \subseteq P$ or $bN \in P$ for all $a,b \in R$ and for all $R$-submodules, $N$ of $M.$

   \end{itemize}
		\end{defn}

Let $P$ be a $v$-classical prime $R$-ideal of $M$ for $v=0,2,3,c$. Then, to simplify, we can say that $P$ is a $v$-classical prime.
  
  \begin{pro}Let $P \lhd_R M $ such that $RM \nsubseteq P$. Let $v=0,2,3,c$. If $P$ is $v$-classical prime $R$-ideal of $M$, then $\widetilde{P}$ is $v$-classical prime ideal of $R.$
  \end{pro}
  \begin{proof}~
      \begin{itemize}
          \item $(v=2):$ Suppose $P$ is $2$-classical prime $R$-ideal of $R.$ Let $A, B$ be left $R$-subgroups of $R$ and $I$ ideal of $R.$ Suppose that $ABI \subseteq \widetilde{P}.$ It follows that $ABI R \subseteq \widetilde{P} \Longrightarrow AB (IR) \subseteq P \Longrightarrow ABI \subseteq P$ since $P$ is $2$-classical prime $R$-ideal of $R.$ Then we have $AI \subseteq P$ or $BI \subseteq P.$  It follows that $A(IR) \subseteq P \Longrightarrow (AI)R \subseteq P$ or $(BI)R \subseteq P$. Hence $AI \subseteq \widetilde{P}$ or $BI \subseteq \widetilde{P}$.
          \item The proof goes  similarly for $v=0,3,c$.
      \end{itemize}
  \end{proof}

  \begin{defn}
      $M$ is said to be $v=0,2,3,c$-classifcal prime $R$-module if $RM \ne 0$ and $\{0 \}$ is a $v$-classical prime
 $R$-ideal of $M$.
 \end{defn}

 \
		  \begin{pro}Let $P \lhd_R M $ such that $RM \nsubseteqq P$. For $v=0,2,3,c$ $\frac{M}{P}$ is $v$-classical prime $R$-module  if and only if $P$ $v$-classical prime ideal of $R$-ideal.
  \end{pro}
  \begin{proof}~
      \begin{itemize}
          \item{($v=0$):} $P$ is $0$-classical prime $\Longleftrightarrow$ for ideals, $A$ and $B$ of $R$, and $R$-submodule, $N$ of $M,$ such that $ABN \subseteq P,$ it follows that, $AN \subseteq P$ or $BN \subseteq P $ $\Longleftrightarrow \frac{ABN}{P} =0$  implies $\frac{AN}{P}=0$ or  $\frac{BN}{P}=0$ $\Longleftrightarrow \{0 \}$ is $0$-classical prime $R$-ideal of $\frac{M}{P} \Longleftrightarrow $ $\frac{M}{P}$ is $0$-classical prime $R$-module.
          \item{($v=2$):} Similar to the previous but by choosing $A$ and $B$ to be left $R$-subgroups of $R$.
          \item{($v=3$):} $P$ is $3$-classical prime $\Longleftrightarrow$ $(aR)(bR)N \subseteq P$ implies $a N \subseteq P$ or $bN \in P$ for all $a,b \in R$ and for all $R$-submodules, $N$ of $M $ $\Longleftrightarrow$ $\frac{(aR)(bR)N}{P}=0$ implies $\frac{aN}{P}=0$ or $\frac{aN}{P}=0$  $\Longleftrightarrow \{0 \}$ is $3$-classical prime $R$-ideal of $\frac{M}{P} \Longleftrightarrow $ $\frac{M}{P}$ is $3$-classical prime $R$-module.
          \item{($v=c$):} It follows similarly with $v=3.$
      \end{itemize}
  \end{proof}

    \begin{thm}     \label{classicalprime}
Let $M$ be an $R$-module and  $P \lhd_R M $. Consider the following:

      \begin{itemize}
          \item[(a)] $P$ is $c$-classical prime,
             \item[(b)] $P$ is $3$-classical prime,
                \item[(c)] $P$ is $2$-classical prime,
                   \item[(d)] $P$ is $0$-classical prime.
      \end{itemize}
      We have $(a) \Longrightarrow (b) \Longrightarrow  (c) \Longrightarrow  (d).   $
    
  \end{thm}

  \begin{proof} ~
  \begin{itemize}
      \item ($(a) \Longrightarrow (b)$)  Suppose $P$ is $c$-classical prime. Let $a, b \in R$ and for any $R$-submodule $N$ of $M,$ we assume that $(aR)(bR)N \subseteq P.$ If $aN \subseteq P$ then we are done. Now, suppose $aN \nsubseteq P.$ We have $(aR)(bR)N =(aRb)RN \subseteq P$. But $RN \subseteq N$, so $(aRb)RN \subseteq (aRb)N \subseteq P.$ Since $P$ is $c$-classical prime then $aN \nsubseteq P$ for $bN \nsubseteq P.$
       \item ($(b) \Longrightarrow (c)$)     Suppose $P$ is $3$-classical prime . Let $A,B$ be  left $R$-subgroups of $R$ and $N$ be an $R$-submodule of $M$ such that $ABN \subseteq P.$  If $AN \subseteq P$, then we are done. Suppose $AN \nsubseteq P$.  Then there exists $a \in A$ such that $aN \nsubseteq P.$ Now, for all $b \in B, (aR)(bR)N=a(Rb)(RN) \subseteq ABN \subseteq P.$ Since $P$ is $3$-classical prime, we have $aN \subseteq P$ or $bN \subseteq P.$ But $aN \nsubseteq P$;hence it follows that $bN \subseteq P$ and this is true for every $b \in B.$ Hence, $BN \subseteq P.$
 \item ($(c) \Longrightarrow (d)$)   Now suppose that $P$ is $2$-classical prime.Let $A$ and $B$ be ideals of $R$ and $N$ be an $R$-submodule of $M$ such that $ABN \subseteq P.$ Clearly, $A$ and $B$ are left $R$-subgroups of $R$. Since $P$ is $2$-classical prime, $AN \subseteq P$ or $BN \subseteq P$ and we are done.
  \end{itemize}

  \end{proof}

  \begin{cor}\label{sd}
      Let $M$ be an $R$-module.
      \begin{itemize}
          \item[(a)] $M$ is $c$-classical prime,
             \item[(b)] $M$ is $3$-classical prime,
                \item[(c)] $M$ is $2$-classical prime,
                   \item[(d)] $M$ is $0$-classical prime,
      \end{itemize}
      We have $(a) \Longrightarrow (b) \Longrightarrow  (c) \Longrightarrow  (d).   $
  \end{cor}
  In the case of a ring, the three classical prime definitions above coincide. However, in the case of a near-ring the four types of classical primes are non-equivalent. We demonstrate this by the following examples:

  \begin{exa}
      Let $R$ be the Klein-4-group, $K=\{ 0,1,2,3 \}$ with multiplication given by

\begin{equation*}
r \circ b=
\begin{cases}
r \thickspace \text{if $b=3$} \\
0 \thickspace \text{if $b \in \{ 0,1,2\}$}
\end{cases}
\end{equation*}

      \begin{table}[!h]
  \centering
  \caption{Table of multiplication  for $\big  (K,+ , \cdot \big )$ }
  \label{tab:table1}
  \begin{tabular}{cccccccccc}
  %  \toprule
   $\cdot $  & $0$ & $1$ &$ 2 $& $ 3 $  \\
  %  \midrule

   $0$ & $0 $ & $ 0$ & $0   $ & $ 0  $   \\
   
    $1$  & $ 0$ & $ 0$ & $  0 $ & $  1  $  \\
    
    $ 2$ & $0 $ & $0 $ & $0  $ & $ 2   $   \\
    
$ 3 $ & $0 $ & $0 $ & $ 0  $ & $ 3  $   \\

%    \bottomrule
  \end{tabular}
\end{table} 

Consider the $R$-module $M=R_R.$ The  non-zero $R$-submodules of $M$ are $\{ 0,1 \}, \{ 0,2 \} $ and $\{ 0,1 ,2,3\}$. Now, $R$ has no non-zero proper ideals. So considering the three possibilities, we note that:
\begin{itemize}
    \item[(i)] $R^2M \ne \{ 0 \}$
      \item[(ii)] $R^2 \{0,1 \} \Longleftrightarrow R\{ 0,1 \}= \{0 \}$
          \item[(iii)] $R^2 \{0,2 \} \Longleftrightarrow R\{ 0,2 \}= \{0 \}$.
    
\end{itemize}
So, $ \{0 \}$ is a $0$-classical prime $R$-ideal and hence $M$ is $0$-classical prime. However, $\{0,2 \}$ is a left $R$-subgroup of $R$ with $\{0,2 \}.\{0,2 \}.\{0,1,2,3 \}= \{0 \}$ but $\{0,2 \}.\{0,1,2,3 \} \ne \{0 \}$.   $M$ is not $2$-classical prime. Also note that $M$ is not $c$-classical prime since $3 R2M=0$ but, $2M \ne 0$ and $3M \ne 0.$ 
  \end{exa}

  \begin{exa}
      Let $R$ be a near-ring defined on $ Z_3=\{ 0,1,2 \}$  by

      \begin{table}[!h]
  \centering
  \caption{Table of multiplication  for $\big  (R,+ , \cdot \big )$ }
  \label{tab:table1}
  \begin{tabular}{cccccccccc}
   % \toprule
   $\cdot $  & $0$ & $1$ &$ 2 $ \\
 %   \midrule

   $0$ & $0 $ & $ 0$ & $0   $   \\
   
    $1$  & $ 0$ & $ 0$ & $  1$  \\
    
    $ 2$ & $0 $ & $0 $ & $2$   \\

   % \bottomrule
  \end{tabular}
\end{table} 

Consider the $R$-module $M=R_R.$ Then $M$ has no proper $R$-submodules; hence $R$ has no proper $R$-subgroups. Since $RRM \ne  \{0 \}$, so $0$-classical prime $R$-ideal implies that $M$ is  $2$-classical prime. However, $(1R)(1R)M=0$ but $1M=1R \ne \{0 \}$. Therefore, $\{0 \}$ is not $3$-classical prime $R$-ideal, and hence $M$ is not $3$-classical prime. Also note that $M$ is not $c$-classical prime since $1 R 1M =0$ but $1 M \ne 0$ and $1 M \ne 0$. Furthermore, $RRR \ne \{0\}$,  so $\{0 \}$ is a $0$-classical prime of $R.$ Note that $R$ is not $c$-classical prime ideal of $R.$
  \end{exa}
\begin{exa} (\cite{pilz2011near}) We refer $GF(3^2)$ as the Galois field of order $3^2.$  Let $R$ be a finite near-field.
\label{ex2}Consider the field ($GF(3^{2})$, $+$, $\cdot$) with
\[GF(3^{2}) := \{0,1,2,x,1+x,2+x,2x,1+2x,2+2x\},\]
where $x$ is a zero of $x^{2}+1 \in \Bbb{Z}_{3}[x]$ with the new multiplication defined as

$$
a \circ b := \left\{\begin{array}{cc}
              a \cdot b     & \mbox{ if $a$  is a square in ($GF(3^{2})$, $+$, $\cdot$)}\\
              a \cdot b^3 & \mbox{ otherwise }
              \end{array}
\right.
$$
This gives the smallest finite  Dickson near-field $R=DN(3,2):=(GF(3^{2})$, $+$, $\circ)$, which is not a field. Here is the table of the new operation $ \circ$ for $DN(3,2)$.

\[
 \begin{array}{r|ccccccccc}
\circ     & 0 & 1             & 2            & x     & 1+x  & 2+x  & 2x   & 1+2x & 2+2x \\ \hline
0         & 0 & 0             & 0            & 0          & 0         & 0         & 0         & 0         & 0\\ 
1         & 0 & 1             & 2            & x     & 1+x  & 2+x  & 2x   & 1+2x & 2+2x \\    
2         & 0 & 2             & 1            & 2x    & 2+2x & 1+2x & x    & 2+x  & 1+x \\    
x    & 0 & x        & 2x      & 2          & 1+2x & 1+x  & 1         & 2+2x & 2+x \\    
1+x  & 0 & 1+x      & 2+2x    & 2+x   & 2         & 2x   & 1+2x & x    & 1 \\    
2+x  & 0 & 2+x      & 1+2x    & 2+2x  & x    & 2         & 1+x  & 1         & 2x \\    
2x   & 0 & 2x       & x       & 1          & 2+x  & 2+2x & 2         & 1+x  & 1+2x \\    
1+2x & 0 & 1+2x     & 2+x     & 1+x   & 2x   & 1         & 2+2x & 2         & x \\    
2+2x & 0 & 2+2x     & 1+x     & 1+2x  & 1         & x    & 2+x  & 2x   & 2     
 \end{array}
\]
 Since $R$ is a near-field,  $R$ has no proper $R$-submodules. Therefore,  has no proper $R$-subgroups. Hence the non-zero $R$-submodules and $R$-subgroup of $R$ is $R$. Note that for every $a, b \in R$ we have $(aRb)R \ne \{0 \}$. It follows that $\{ 0\}$ is a $c$-classical prime $R$-ideal. Thus $R_R$ is a $c$-classical prime $R$-module.\\
\end{exa}

\begin{exa}

Let $R$ be a near-ring defined on $ Z_6=\{ 0,1,2 ,3,4,5\}$  by

      \begin{table}[!h]
  \centering
  \caption{Table of multiplication  for $\big  (R,+ , \cdot \big )$ }
  \label{tab:table1}
  \begin{tabular}{cccccccccc}
   % \toprule
   $\cdot $  & $0$ & $1$ &$ 2 $ &$ 3$ &$ 4$ &$ 5$  \\
   % \midrule

    $0$ & $0$ &$ 0 $ &$ 0$ &$ 0$ &$ 0$  &$ 0$  \\
   
    $1$ & $3$ &$ 5 $ &$ 1$ &$ 3$ &$ 5$ &$ 1$   \\
    
    $2$ & $0$ &$ 4 $ &$ 2$ &$ 0$ &$ 4$ &$ 2$   \\
   
    $3$ & $3$ &$ 3 $ &$ 3$ &$ 3$ &$ 3$ &$ 3$   \\
    
    $4$ & $0$ &$ 2 $ &$ 4$ &$ 0$ &$ 2$ &$ 4$   \\
$5$ & $3$ &$ 1 $ &$ 5$ &$ 3$ &$ 1$ &$ 5$  \\
  %  \bottomrule
  \end{tabular}
\end{table} 

Let $M=R_R.$ Then, the non-zero $R$-submodules of $M$ are $\{0,3 \}$ and  $\{0, 1, 2, 3, 4, 5 \}$  Note that in the Definition \ref{new} $(d)$ does not imply $(c)$ in general. To see this, consider the $R$-submodules $N=\{0,3 \}$.  We have that for all $a,b \in R \: (aR)(bR)= \{0\} \Longrightarrow aN= \{0 \}$ or $bN = \{0 \}.$ Hence the $(c)$ is satisfied. But, $(d)$ is not satisfied. In fact there exist $a=3$ and $b=3$ such that $aRb=0$ but $aN \ne 0$ and $bN \ne 0.$

    % Let consider $R= \mathbb{Z}_5= \{0,1,2,3,4 \}.$ Define multiplication as $a . b=a+b+1$ modulo $5.$ It is clear that $(R,.)$ is a semigroup. Let $a=2$ and $b=3$ two elements of $R.$ We have that $2.R.3 = \{ 2.r.3: r \in R \}=\{ 0,1,2,3,4\}$. But $(2.R)(3.R)= \{ 2.r: r \in R \} \{ 3.r: r \in R \}= \{0,1, 4 \}.$ Thus $(aRb) \ne (aR)(bR).$ This explains the differences in the definition of hypothesis of %$3$-classical prime and $3$-classical prime.
 \end{exa}
 \begin{rem} Let $v=0,2,3,c.$
    When we consider the near-ring modules $R_R$. The examples of $v$-classical $R$-modules work similarly for $v$-classical prime ideals of near-ring $R$.
 \end{rem}
\begin{pro}Let $R$ be a near-field. 
    Any proper $R$-ideal of $R^n$ is a $c$-classical prime $R$-ideals of $R^n$.
\end{pro}
\begin{proof}
    Let $P$ be any  proper $R$-ideals of $R^n$. Then, by Theorem \ref{thm:ege}  $P$ is of the form  $ S_1 \times S_2 \times \cdots \times S_n$ where $S_i= \{ 0 \}$ or $S_i=R$ for $i=1, \ldots,n$. Assume without loss of generality that $S_j= \{0 \}$ for some fixed $j$ where $1 \leq i <j<n$. Let $N$ be any $R$-submodule of $R^n$. Then by Theorem \ref{thm:ege1}, we have 
			 $N=\bigoplus_{i=1}^\ell Ru_i$
			for some non-zero vectors $u_1,\ldots,u_\ell$ with mutually disjoint supports. Let $a, b \in R$. Suppose $a (\bigoplus_{i=1}^\ell Ru_i ) \nsubseteq P$ and $b (\bigoplus_{i=1}^\ell Ru_i ) \nsubseteq P$. We have $a (\bigoplus_{i=1}^\ell Ru_i ) \nsubseteq P$  implies  there exists $r \in R$ at the $j$-component  of $\bigoplus_{i=1}^\ell Ru_i$ such that $(ar)_j \ne 0.$ Similarly  $b (\bigoplus_{i=1}^\ell Ru_i ) \nsubseteq P$  implies  there exists $s \in R$ at the $j$-component  of $\bigoplus_{i=1}^\ell Ru_i$ such that $(bs)_j \ne 0.$ It follows that $(arbs)_j \ne 0.$ Thus $(aRb)\bigoplus_{i=1}^\ell Ru_i \nsubseteq P$.  
\end{proof}
\begin{cor}Let $R$ be a near-field. 
    $R^n$ is a $c$-classical prime $R$-module.
\end{cor}
 
 \begin{proof}
     It follows from the previous Proposition.
 \end{proof}

  \begin{pro}\label{sss}
      Let $R$ be a near-ring with identity $1$,  and let $P$ be an $R$-ideal of $M.$ Then $P$ is $2$-classical prime if and only if $P$ is $3$-classical prime.
  \end{pro}
  \begin{proof}
      The fact that if $P$ is $3$-classical prime implies that it is $2$-classical prime has already been proven  in Theorem \ref{classicalprime}. Now suppose that $P$ is a $2$-classical prime $R$-ideal. Let $a,b  \in R$ be an $R$-submodule of $M$ such that $(aR)(bR)N \subseteq P.$ Then, for every $n \in N,$ we have that $(Ra)(Rb)(Rn)=R (aR)(bR)n \subseteq RP \subseteq P.$ Since, $Ra$ and $Rb$ are  left $R$-subgroups of $R$ and $Rn$ is an $R$-submodule of $M$, it follows from the fact that $P$ is $2$-classical, that $(Ra)(Rn) \subseteq P$ or $(Rb)(Rn) \subseteq P.$ In particular, since $1 \in R,$ it follows $1.a.1.n=an \in P$ or $1.b.1.n= bn \in P$ for every $n \in N.$ So $aN \subseteq P$ or $bN \subseteq P.$ Therefore $P$ is $3$-classical prime.
  \end{proof}
\begin{thm} Let $v=0,2,3,c$. 
    Every $v$-prime near-ring module is a $v$-classical near-ring module.
\end{thm}

    In general, for $v=0,2,3,c$ every $v$-classical prime near-ring module need not be $v$-prime.   We demonstrate this with the following examples.

     \begin{exa}
     \label{ss}
      Let $R$ be a near-ring defined on $ Z_4=\{ 0,1,2,3 \}$  by

      \begin{table}[!h]
  \centering
  \caption{Table of multiplication  for $\big  (R,+ , \cdot \big )$ }
  \label{tab:table1}
  \begin{tabular}{cccccccccc}
 %   \toprule
   $\cdot $  & $0$ & $1$ &$ 2 $ & $3$ \\
 %   \midrule

   $0$ & $0 $ & $ 0$ & $0   $ & $0$  \\
   
    $1$  & $ 0$ & $ 1$ & $  0$ & $0$   \\
    
    $ 2$ & $0 $ & $0 $ & $0$ & $0$    \\
        $ 3$ & $0 $ & $1 $ & $0$ & $1$    \\
    
  %  \bottomrule
  \end{tabular}
\end{table} 

Let $M=R_R.$ Then, the non-zero $R$-submodules of $M$ are $ \{ 0,1\}, \{0,2 \}$ and $M.$ Checking all the possibilities according to the Definition of a $3$-classical prime module, we have the following observations:
\begin{itemize}
    \item The following are all non-zero products: $1R1R \{0,1 \}, 1R1R M, 1R3R \{0,1 \}, 1R3R M,$  $ 3R3R \{0,1 \}, 3R3R M, $
    $3R1R \{0,1 \}$ and $3R1R M.$
    \item $1R2R \{0,1 \}=2R2R \{0,1 \}=2R1R \{0,1 \}= 2R3R \{0,1 \}=3R2R \{0,1 \}= \{0 \}$ and in each case $2. \{0,1 \}= \{0 \}.$
\item $ 1R1R \{0,2 \}= 1R2R \{0,2 \} =2R3R \{0,2 \}=2R2R \{0,2 \}=2R1R \{0,2 \}=2R3R \{0,2 \}=3R3R \{0,2 \}= 3R2R \{0,2 \}= 3R1R \{0,2 \}=0$ and in each case, any one of $1. \{0,2 \}= \{0 \}$ or  $2. \{0,2 \}= \{0 \}$  or  $2. \{0,2 \}= \{0 \}$ is satisfied. 
\item $1R2RM=2R2RM=2R1RM=2R3RM=3R2RM= \{0\}$ and in all cases $2 . M= \{0 \}$.
\end{itemize}
Hence $M$ is  $3$-classical prime $R$-module. However, $M$ is not $3$-prime since $1R2= \{0 \}$ but $1M \ne  \{0 \}$ and $ 2 \notin \{0 \} $. 

In the same way, we show that $M$ is a $c$-classical prime $R$-module.

\begin{itemize}
      \item The following are all non-zero products: $1R1 \{0,1 \}, 1R1 M, 1R3R \{0,1 \}, 1R3 M,$  $ 3R3 \{0,1 \}, 3R3 M, $
    $3R1 \{0,1 \}$ and $3R1 M.$
    \item $1R2 \{0,1 \}=2R2 \{0,1 \}=2R1 \{0,1 \}= 2R3 \{0,1 \}=3R2 \{0,1 \}= \{0 \}$ and in each case $2. \{0,1 \}= \{0 \}.$
\item $ 1R1 \{0,2 \}= 1R2 \{0,2 \} =2R3 \{0,2 \}=2R2 \{0,2 \}=2R1 \{0,2 \}=2R3 \{0,2 \}=3R3 \{0,2 \}= 3R2 \{0,2 \}= 3R1 \{0,2 \}=0$ and in each case, any one of $1. \{0,2 \}= \{0 \}$ or  $2. \{0,2 \}= \{0 \}$  or  $2. \{0,2 \}= \{0 \}$ is satisfied. 
\item $1R2M=2R2M=2R1M=2R3M=3R2M= \{0\}$ and in all cases $2 . M= \{0 \}$.
\end{itemize}

Hence $M$ is a $c$-classical prime $R$-module. However, $M$ is not $c$-prime since $1.2= 0 $ but $1M \ne  \{0 \}$ and $ 2 \ne 0 $. 
  \end{exa}
  \begin{exa}
      If $R$ is a near-ring with identity, then we know that a $2$-prime module and a $3$-prime module are equivalent. From Proposition \ref{sss}, the same is also true for $2$- and $3$-classical prime modules. So, if $R$ has identity and since the previous examples demonstrate that a $3$- classical prime module need not be $3$-prime, it follows that a $2$-classical prime module need not be $2$-prime.
  \end{exa}

  \begin{exa}
      In the Example \ref{ss}, we have a near-ring module, $M=R_R,$ which  is $3$-classical prime. By Corollary \ref{sd}, it is $0$-classical prime. However, in the same example, $\{ 0,1 \}$ is an ideal and $ \{0,2 \}$ is an $R$-submodule such that $ \{0,1  \}\{0,2 \}=0$ but neither $\{0,1 \}M=0$ nor $\{0,2 \} \subseteq \{0 \}.$ So $M$ is not $0$-prime.
  \end{exa}
The following theorem provides the characterization of $0$-classical prime $R$-ideal of $M.$
  \begin{thm}
  \label{prod}
  Let $M$ be an $R$-module and let $P\lhd_R M. $ Then the following statements are equivalents:

  \begin{itemize}
      \item[(i)] $P$ is $0$-classical prime
      \item[(ii)] For all ideals, $A$ and $B$ of $R$, and every $m \in M$ such that $AB (Rm) \subseteq P,$ it follows that $A(Rm) \subseteq P$ or $B(Rm) \subseteq P.$
       \item[(iii)] For every $0 \ne \overline{m} \in \frac{M}{P}, (0:R \overline{m})$ is a $0$-prime ideal of $R$.
       \item[(iv)] For every $m \in M \backslash P, (P:Rm)$ is a $0$-prime ideal of $R,$ and $(P:M)$ is a $0$-prime ideal of $R$.
  \end{itemize}
      
  \end{thm}
  \begin{proof}~
  \begin{itemize}
      \item $(i) \Longrightarrow (ii)$ Let $A$ and $B$ be ideals of $R$, and let $m \in M$ such that $AB(Rm) \subseteq P.$ Since $P$ is $0$-classical prime and $Rm$ is an $R$-submodule of $M,$ by definition it follows that $A(Rm) \subseteq P$ or $B(Rm) \subseteq P.$
          \item $(ii) \Longrightarrow (iii)$ Suppose that $0 \ne  \overline{m} \in \frac{M}{P}.$ Let $A$ and $B$ be ideals of $R$ such that $AB \subseteq (0:R\overline{m}) =  \{ a \in R: aR\overline{m}=0 \}$. Then $AB(R\overline{m})=0 \Longrightarrow AB(Rm) \subseteq P.$ From $(ii)$, it follows that $A(Rm) \subseteq P$ or $B(Rm) \subseteq P$. Hence $A(R\overline{m})=0$ or  $B(R\overline{m})=0 \Longrightarrow A \subseteq (0:R\overline{m}) $ or $B \subseteq (0:R\overline{m}) $.
 \item  $(iii) \Longrightarrow (iv)$ The first part of $(iv)$ is simply a restatement of $(iii).$ Now let $A$ and $B$ be ideals of $R$ such that $AB\subseteq (P:M)$ which implies that $ABM \subseteq P.$ We need to show that $AM \subseteq P$ or $BM \subseteq P.$ Suppose that $BM \nsubseteq P.$ Since $P lhd_R M,$ we know that $AP \subseteq P$ and $BP \subseteq P.$ So consider $n \in M \backslash P$ such that $Bn \nsubseteq P.$ Now $AB (Rn) \subseteq ABM \subseteq P$ implies that $AB \subseteq (P:Rn).$ Since $(P:Rn)$ is $0$-prime ideal of $R$, it follows that $A \subseteq (P:Rn)$ or $B \subseteq (P:Rn)$ whence $A(Rn)\subseteq P$ or $B(Rn)\subseteq P$. But, $Bn \nsubseteq$ implies $B (Rn) \nsubseteq P.$ So, $A(Rn) \subseteq P.$

 Now let $ m \in M \backslash P$ where $m \ne n.$ Then, either $(P:Rn) \subseteq (P:Rm)$ or $(P:Rm) \subseteq (P:Rn).$ If $(P:Rn) \subseteq (P:Rm)$, then, since $A \subseteq (P:Rn),$ we have $A \subseteq (P:Rm)$ implying that $Am \subseteq A(Rm) \subseteq P.$

 If $(P:Rm) \subseteq (P:Rn)$, then $AB(Rm) \subseteq ABM \subseteq P \Longrightarrow AB \subseteq (P:Rm).$ Since $(P:Rm)$ is $0$-prime, we have $A \subseteq (P:Rm)$ or $B \subseteq (P:Rm).$ But, $B \subseteq (P:Rm) \Longrightarrow B \subseteq (P:Rn) \Longrightarrow B (Rn) \subseteq P$ is a contradiction. So $A(Rm) \subseteq P \Longrightarrow Am \subseteq P.$ Hence $AM \subseteq P.$
 
\item   $(iv) \Longrightarrow (i)$  Let $A$ and $B$ be ideals of $R$ and $N$ be an $R$-submodule of $M$ such that $ABN \subseteq P. $ Then $AB(RN)\subseteq ABN \subseteq P$. So, for all $n \in N, AB(Rn) \subseteq P $ implies $AB \subseteq (P: Rn).$ If $n  \in P,$ then we know that $An \subseteq P$ and $Bn \subseteq P.$ If $n \notin P,$ then  from $(iv), (P:Rn)$ is $0$-prime. So $AB \subseteq (P:Rn)  \Longrightarrow A \subseteq (P:Rn)$ or $B \subseteq (P:Rn) \Longrightarrow A (Rn) \subseteq P$ or $B (Rn) \subseteq P \Longrightarrow An \subseteq P$ or $Bn \subseteq P.$ Hence, for all $n \in N, An \subseteq P$ or $An \subseteq P$ or $Bn \subseteq P.$ So $AN \subseteq P$ or $BN \subseteq P$ and we are done.
 \end{itemize}
  \end{proof}
Similarly, the following result gives the characterization of $2$-classical prime $R$-ideal of $M.$
    \begin{pro} Let $M$ be an $R$-module and let $P\lhd_R M. $ Then the following statements are equivalents:

  \begin{itemize}
      \item[(i)] $P$ is $2$-classical prime.
      \item[(ii)] For all left $R$-subgoups, $A$ and $B$ of $R$, and every $m \in M$ such that $AB (Rm) \subseteq P,$ it follows that $A(Rm) \subseteq P$ or $B(Rm) \subseteq P.$
       \item[(iii)] For every $0 \ne \overline{m} \in \frac{M}{P}, (0:R \overline{m})$ is a $2$-prime ideal of $R$.
       \item[(iv)] For every $m \in M \backslash P, (P:Rm)$ is a $2$-prime ideal of $R,$ and $(P:M)$ is a $2$-prime ideal of $R$.
  \end{itemize}
      
  \end{pro}
  The proof is similar to the proof of the previous theorem.
 \begin{pro} Let $M$ be an $R$-module and let $P\lhd_R M. $ Then the following statements are equivalents:

  \begin{itemize}
      \item[(i)] $P$ is $3$-classical prime.
      \item[(ii)] For all $a,b \in R$ and every  $R$-submodule $N$ of $M$ such that $aRb (N) \subseteq P,$ it follows that $aN \subseteq P$ or $bN\subseteq P.$
      \item[(iii)] For all $a,b \in R$ and every   $ m \in M$ such that $(aR)(bR)m  \subseteq P,$ it follows that $aRm \subseteq P$ or $b Rm \subseteq P.$
       \item[(iv)] For every $0 \ne \overline{m} \in \frac{M}{P}, (0:R \overline{m})$ is a $3$-prime ideal of $R$.
       \item[(v)] For every $m \in M \backslash P, (P:Rm)$ is a $3$-prime ideal of $R,$ and $(P:M)$ is a $2$-prime ideal of $R$.
  \end{itemize}
      
  \end{pro}

 \begin{proof}~
  \begin{itemize}
      \item $(i) \Longrightarrow (ii)$ Let $a,b  \in R$ and let $N$ be a $R$-submodule of $M$ such that $aRb(N) \subseteq P.$ Then $(aR)(bR) N \subseteq a Rb (N)  \subseteq P$. Since, $P$ is $3$-classical prime, $aN \subseteq P$ or $bN \subseteq P.$
 \item  $(ii) \Longrightarrow (iii)$ Let $a,b  \in R$ and let $m \in M$ such that $(aR)(bR)m \subseteq P$ which implies $(aRb)(Rm) \subseteq P.$ Then, from $(ii)$, it follows that $aRm \subseteq P$ or $bRm \subseteq P.$

\item   $(iii) \Longrightarrow (iv)$ Let  $0 \ne \overline{m} \in \frac{M}{P} $ and let $a,b \in R$ such that $aRb \subseteq (0: R \overline{m}).$ Then $(aRb)(R \overline{m})=0 \Longrightarrow (aRbR)m \subseteq P$. From $(iii)$, it follows that $aRm \subseteq P$ or $bRm \subseteq P$. Hence $a (R \overline{m})=0$ or $b (R \overline{m})=0 \Longrightarrow a \in (0: R \overline{m})$ or  $ b \in (0: R \overline{m})$

\item   $(iv) \Longrightarrow (v)$  The first part of $(v)$ is simply a restatement of $iv$. Now, let $a, b \in R$ such that $aRb \subseteq  (P:M)$ which implies that $(aRb)M \subseteq P$. We need to show that $aM \subseteq P$ or $bM \subseteq P.$ The rest of the proof is simply an adaptation of the third part in the proof of Theorem \ref{prod}.

\item   $(v) \Longrightarrow (i)$ It is similar to the fourth of part of the proof of Theorem \ref{prod}.
 \end{itemize}

 \end{proof}

 \begin{pro} Let $M$ be an $R$-module and let $P\lhd_R M. $ Then the following statements are equivalent:

  \begin{itemize}
      \item[(i)] $P$ is $c$-classical prime.
      \item[(ii)] For all $a,b \in R$ and every  $R$-submodule $N$ of $M$ such that $aRb (N) \subseteq P,$ it follows that $aN \subseteq P$ or $bN\subseteq P.$
      \item[(iii)] For all $a,b \in R$ and every   $ m \in M$ such that $(aRb)m  \subseteq P,$ it follows that $aRm \subseteq P$ or $b Rm \subseteq P.$
       \item[(iv)] For every $0 \ne \overline{m} \in \frac{M}{P}, (0:R \overline{m})$ is a $c$-prime ideal of $R$.
       \item[(v)] For every $m \in M \backslash P, (P:Rm)$ is a $c$-prime ideal of $R,$ and $(P:M)$ is a $2$-prime ideal of $R$.
  \end{itemize}
      
  \end{pro}

  The proof is similar to the proof of the previous Proposition.

  \begin{pro}
      Let $v=0,2,3,c$ and suppose that $P$ is a $v$-classical prime $R$-ideal of an $R$-module $M.$ Then, for any $R$-ideal, $N$ of $M, (P:N)$ is $v$-prime ideal of $R$.
  \end{pro}
\begin{proof}~
   \begin{itemize}
      \item $(v=0):$ Let $A$ and $B$ be ideals of $R$ such that $AB \subseteq (P:N).$ Then $ABN \subseteq P, $ and since $P$ is $0$-classical prime, we have that $AN \subseteq P$ or $BN \subseteq P.$ So, $A \subseteq (P:N)$ or $B \subseteq (P:N).$
      \item $(v=2):$ Similar to the case $v=0,$ but by choosing $A$ and $B$ to be left $R$-subgroups of $R.$
      \item $(v=3):$ Let $a, b \in R$ such that $aRb \subseteq (P:N).$ Then, $(aRb)N \subseteq P.$ Now $(aR)(bR)N=(aRb)(RN) \subseteq (a R b)N \subseteq P.$ Since, $P$ is $3$-classical prime. It follows that $aN \subseteq P$ or $bN \subseteq P;$ whence $a \in (P:N)$ or $b \in (P:N)$.
      \item $(v=c):$ Similar to the case $v=3.$
  \end{itemize}

\end{proof}

  \begin{pro}
      Let $v=0,2,3,c$ and suppose that $P$ is $v$-prime $R$-ideal of an $R$-module $M.$ Then,  $P$ is $v$-classical $R$-ideal of $M.$
  \end{pro}
\begin{proof}~
   \begin{itemize}
       \item  $(v=0):$ Let $A$ and $B$ be ideals of $R$ and $N$ be an $R$-submodule of $M$ such that $ABN \subseteq P.$ Then, for each $n \in N, A(Bn) \subseteq ABN \subseteq P.$ Since, $Bn$  is a submodule of $M$ and $P$ is $0$-prime, it follows that $AM \subseteq P$ or $Bn \subseteq P,$ and hence $AN \subseteq P$ or $BN \subseteq P.$
       \item  $(v=2):$ Let $A$ and $B$ be left $R$-subgroups of $R$ and $N$ be an $R$-submodule of $M$ such that $ABN \subseteq P.$ The rest of the proof follows as in the previous case. 
       \item  $(v=3):$ Let $a, b \in R$ and $N$ be an $R$-submodule of $M$  such that $(aR)(bR)N \subseteq P.$ Then for each $n \in bRN,$ we have $aRn \subseteq (aR)(bRN) \subseteq P.$ Since, $P$ is $3$-prime, we have that $aM \subseteq P$ or $n \in P.$ If $aM \subseteq P$, then $aN \subseteq P$ and we are done. If $n \in P,$ then $bRN \subseteq P.$ So for each $m \in N, bRm \subseteq P.$ Again, since $P$ is $3$-prime, we get $bM \subseteq P$ or $ m \in P.$ If $bM \subseteq P$ then $bN \subseteq P.$ If $m \in P,$ then $N \subseteq P.$ But, $bN \subseteq N.$ So, in either case $bN \subseteq P.$
       \item $(v=c):$ Similar to the case $v=3.$
   \end{itemize}
  \end{proof}
    \begin{cor}
          Let $v=0,2,3,c$ and   $P$ is a $v$-prime $R$-ideal of an $R$-module $M$.  Then, for any $R$-ideal, $N$ of $M, (P:N)$ is $v$-classical prime ideal of $R$
    \end{cor}

\begin{defn}
    Let $R$ be a near-ring and $M$ be an $R$-module. Then a non-empty set $S \subseteq M \backslash \{ 0 \}$ is called a:
    \begin{itemize}
        \item[(a)] classical $m_0$-system if for all ideals, $A$ and $B$ of $R$, and all $R$-submodules $K$ and $L$ of $M$ such that $(K+AL) \cap S \ne \emptyset$ and $(K+BL) \cap S \ne \emptyset $, it follows that $(K+ABL) \cap S \ne \emptyset$.
           \item[(b)] classical $m_2$-system if for all left $R$-subgroups, $A$ and $B$ of $R$, and all $R$-submodules $K$ and $L$ of $M$ such that $(K+AL) \cap S \ne 0$ and $(K+BL) \cap S \ne \emptyset $, it follows that $(K+ABL) \cap S \ne \emptyset $.

             \item[(c)] classical $m_3$-system if for all $ a,b \in R$, and all $R$-submodules $K$ and $L$ of $M$ such that $(K+aL) \cap S \ne \emptyset $ and $(K+bL) \cap S \ne \emptyset$, it follows that $(K+(aR)(bR)L) \cap S \ne \emptyset $.

               \item[(d)] classical $m_c$-system if for all $ a,b \in R$, and all $R$-submodules $K$ and $L$ of $M$ such that $(K+aL) \cap S \ne \emptyset$ and $(K+bL) \cap S \ne \emptyset$, it follows that $(K+abL) \cap S \ne \emptyset $.
        
    \end{itemize}
\end{defn}

\begin{thm}
    Let $R$ be a near-ring and $M$ be an $R$-module. Then for $v=0,2,3,c$ we have $P \lhd_R M$ is $v$-classical prime if and only if $M \backslash P$ is a classical $m_v$-system.
\end{thm}
\begin{proof}~
    \begin{itemize}
      
        \item $(v=0):$ Suppose $P$ is $0$-classical prime and let $S=M \backslash P.$ Let $A$ and $B$ be ideals of $R$, and $K$ and $L$ be $R$-submodules of $M$ such that $(K+AL) \cap S \ne \emptyset$ and $(K+BL) \cap S \ne \emptyset$. If $(K+AL) \cap S \ne \emptyset$, then $ABL \subseteq P.$ Since $P$ is $0$-classical prime, $AL \subseteq P$ or $BL \subseteq P.$ Hence, it follows that $(K+AL) \cap S = \emptyset$ or $(K+BL) \cap S = \emptyset$ which is a contradiction. So, $M \backslash P$ is a classical $m_0$-system.
        
        On the other hand, let $S=M \backslash P$ be a classical $m_0$-system. Suppose $ABL \subseteq P$  where $A$ and $B$ are ideals of $R$ and $L$ is an $R$-submodule of $M.$ If $AL \nsubseteq P$ and $BL \nsubseteq P,$ then $AL \cap S \ne \emptyset$ and $BL \cap S \ne \emptyset$. However, since $ABL \subseteq P,$ we have $ABL \cap S = \emptyset$ which contradicts that $S$ is an $m_0$-system. Hence $AL \subseteq P$ or $BL \subseteq P$ implies that $P$ is $0$-classical prime.
        \item $(v=2):$ Similar to the case $v=0.$
        \item $(v=3):$ Suppose $P$ is  $3$-classical prime and let $S=M \backslash P.$ Let $a, b \in R$ and $K$ and $L$ be $R$-submodules of $M$ such that $(K+aL) \cap S \ne \emptyset $ and $(K+bL) \cap S \ne \emptyset$.  If $(K+(aR)(bR)L) \cap S = \emptyset $ then $(aR)(bR)L \subseteq P.$ Since $P$ is $3$-classical prime, $aL \subseteq P$ or $bL \subseteq P$. Hence, it follows that $(K+aL) \cap S = \emptyset$ or $(K+bL) \cap S = \emptyset$ which is a contradiction. So, $M \backslash P$ is a classical $m_3$-system.

        Conversely, let $S=M \backslash P$ be a classical $m_3$-system. Suppose $(aR)(bR)L \subseteq P$ where $a, b \in R$ and $L$ is an $R$-submodule of $M.$  If $aL \nsubseteq P$ and $bL \nsubseteq P$, then $aL \cap S \ne \emptyset.$ However, since $(aR)(bR)L \subseteq P$, we have $(aR)(bR)L \cap S = \emptyset,$ which contradicts that $S$ is an $m_3$-system. Hence  $aL \subseteq P$ or $bL \subseteq P$ implies that $P$ is $3$-classical prime.
 \item $(v=c):$ Similar to the case $v=3.$
    \end{itemize}
\end{proof}
  
In what follows, we would like to discuss few properties of the annihilator.
\begin{defn}
    Let $P \subseteq M.$ The left annihilator of $P$ in $R$ is defined by $\Ann(P)= \{ r \in R : rP=0 \}$
\end{defn}
  It is clear that $\Ann(P)$ is a left ideal of $R$. It is shown that if $P \leq_R R $ then $\Ann(P)\lhd R.$

  \begin{pro} Let $v=0,2,3,c.$
	    If $P$ is $v$-classical $R$-ideals of $R_R$ then $\Ann(P)$ is a $v$-classical prime ideal of $R.$
	\end{pro}
 \begin{proof}~
     \begin{itemize}
         \item For $v=2$:  Assume $0 \ne P \lhd_R M$ such that $P$ is $2$-classical prime. Let $A,B$ be left $R$-subgoups of $R$ and $I$ ideals of $R. $ We have $ABI \subseteq \Ann(P).$ Then $ABIP =0$ implies $AB(IP)=0$, which implies $AB=0.$ Since $P$ is $2$-classical prime $R$-ideal of $R_R$, $ABI=0 \Longrightarrow AI=0$ or $BI=0 \Longrightarrow  (AI)P=0$ or $(BI)P=0 \Longrightarrow AI \subseteq l(P)$ or $BI \subseteq \Ann(P).$
         \item For $v=3$: Similar to the case $v=3.$
     \end{itemize}
 \end{proof}

 \begin{pro}
     Let $M$ be $3$-prime $R$-module. $M$ is faithful near-ring module if and only if $\Ann(\gen(m)) =0$ for all $0 \ne m \in M.$
 \end{pro}

 \begin{proof}
     Suppose $M$ is $3$-prime, so $ \{0 \}$ is $3$-prime. By definition $\Ann(\gen(m)) = \{r  \in R: r \gen(m)=0 \}$. Let $r  \in R$ such that $r \gen(m)=0$. It follows that $rRm=0$ since $\gen(m)=Rm.$ Hence we have $rM=0$ or $m=0$ since $ \{0 \}$ is $3$-prime. Since $m \ne 0$ then we have $rM=0$. This implies that $r=0$ because $M$ is faithful. Thus $\Ann(\gen(m))=0.$ Conversely suppose that $r \in \Ann(\gen(m)) =0$ and $rM=0.$ We have $\Ann(\gen(m)) =0 \Longrightarrow r \gen(m)=0 \Longrightarrow r Rm =0$. Since $M$ is $3$-prime we have $rM=0$ and since $\Ann(\gen(m))=0$ then $r=0.$
 \end{proof}
 \begin{pro}
     If $M$ is $0,2$-classical prime. Then for all faithful $R$-submodules $A$ and $B$ of $M$ where $\Ann(B) A \ne 0$ we have $\Ann(A)=0.$
 \end{pro}
 \begin{proof}
     Suppose $M$ is $2$-classical prime and let $0 \ne A,B \leq_R M$. Then $\Ann(A)A=0$ and $\Ann(B)B=0.$ Since  $M$ is $2$-classical prime and $\Ann(B) A \ne 0$ then $\Ann(A)A=0$. It follows that $\Ann(A)=0$ because $A$ is faithful.
     \end{proof}
     \begin{pro} Let $R$ be a near-field.
         Consider the $R$-module $R^n$. Let $0 \ne S \lhd _R R^n$. Then $\Ann(S)=0.$
     \end{pro}
     \begin{proof}
         We have $\Ann(S)S=0$. By Theorem \ref{thm:ege}
all the subspaces (or $R$-ideals) of $R^n$ are all of the form $ S_1 \times S_2 \times \cdots \times S_n$ where $S_i= \{ 0 \}$ or $S_i=R$ for $i=1, \ldots,n$. Suppose, without loss of generality, that at the position $j$, $S_j=R$. So $\Ann(S)R=0 $ implies that $\Ann(S)=0.$
     \end{proof}

     \begin{pro}[Corollary~15 in~\cite{djagbajan}]
     \label{rp}
			Let $R$ be a proper near-field and $T\subseteq R^n$. Then,
			$T$ is an  $R$-submodule if and only if $T=\bigoplus_{i=1}^\ell Ru_i$
			for some nonzero vectors $u_1,\ldots,u_\ell$ with mutually disjoint supports. Futheremore,  $T=\bigoplus_{i=1}^\ell Ru_i$ implies that  $T=\sum_{i=1}^\ell Ru_i$ and $x_i+x_j=x_j+x_i$ for all $x_i \in Ru_i$ and $x_j \in Ru_j.$
     \end{pro}
     \begin{proof}
         Assume that $T=\sum_{i=1}^\ell Ru_i$ where $ \lbrace Ru_i : \thickspace  i =1, \cdots, k \rbrace$ are  ideals of $R^\ell.$ Let $x_i \in Ru_i$ and $ x_j \in Ru_j$ such that $i \neq j.$ Indeed $x_i+x_j-x_i-x_j \in Ru_i$ since $x_j-x_i-x_j \in Ru_i$ since $(Ru_i,+)$ is a normal subgroup of $(R^\ell,+).$ Also $x_i+x_j-x_i-x_j \in Ru_j$ since $x_i+x_j-x_i \in Ru_j.$ It follows that $x_i+x_j-x_i-x_j \in Ru_i \cap Ru_j \subseteq Ru_j \cap \sum _{i \in I, i \ne j}Ru_i= \lbrace 0 \rbrace$. But $Ru_i \cap Ru_j \neq \emptyset$ since $ 0 \in Ru_i \cap Ru_j,$ so   $Ru_i \cap M_j=\lbrace 0 \rbrace. $ It follows that $x_i+x_j-x_i-x_j \in Ru_j \cap Ru_j=\lbrace 0 \rbrace.$ Hence $x_i+x_j=x_j+x_i.$
     \end{proof}
     \begin{pro} 		Let $R$ be a proper near-field. For any $u_i, r_j \in R$ such that $r_jRu_i=0$ for all $j\ne i$,  we have
         $\bigoplus_{i=1}^k \Ann(Ru_i) = \Ann(\bigoplus_{i=1}^k Ru_i ) $
     \end{pro}
     \begin{proof}
         Let $x \in \bigoplus_{i=1}^k\Ann(Ru_i)  $. Then there exist $r_i \in 		\Ann(Ru_i) $ such that $x=\bigoplus_{i=1}^kr_i$. In fact, $r_i \in \Ann(Ru_i) $ implies $r_iRu_i=0.$ Let's first show that $x (\bigoplus_{i=1}^kRu_i)=x (\sum_{i=1}^kRu_i)=\sum_{i=1}^k(xRu_i)$.

         Suppose $I=\{1, \ldots, k \}$ where $k \in \mathbb{N}.$ We proceed by induction on $k$. For $k=2,$ let $T= Ru_i \oplus Ru_2. $ So $T=Ru_i+Ru_2$ and $Ru_i \cap Ru_2= \lbrace 0 \rbrace$ where $Ru_i$ and $Ru_2$ are $R$-idea of $R^k.$ Let  $m \in T$. Then $m=x_1+x_2$ for  $x_1 \in Ru_i$ and $x_2 \in Ru_2.$ We need to show that $x(x_1+x_2) = xx_1+xx_2.$  It suffices to show that $x(x_1+x_2)-xx_1-xx_2 \in Ru_i \cap Ru_2.$ Since $Ru_2$ is $R$-ideals of $T,$  $x(x_1+x_2)-xx_1 \in Ru_2$. Then $x(x_1+x_2)-xx_1-xx_2 \in Ru_2.$ Also we have $x(x_2+x_1)-xx_2 \in Ru_i.$ So $x(x_2+x_1) -xx_2-xx_1 \in  Ru_i$. By Proposition \ref{rp} we have $ x(x_2+x_1) -xx_2-xx_1  \in Ru_1 \cap Ru_2 $.

Assume that if $m=\sum_{i =1}^{k-1} x_i$ where $x_i \in Ru_i$, then $xm=x \big (\sum_{i =1}^{n-1}x_i \big ) = \sum_{i =1}^{k-1}( xx_i)$. Let $m \in T ,$ $x \in R$ and suppose $m=x_1+ \ldots+x_k$ where $x_i \in Ru_i$. By Proposition \ref{rp} we have
$x(x_1+\ldots+x_k) - xx_1- xx_2-\ldots-xx_k =  x(x_2+\ldots+x_k+x_1) -x(x_2+x_3+\ldots+x_k)-xx_1 \in Ru_1.$
Also, $
x(x_1+x_3+\ldots+x_k+x_2) -x(x_1+x_3+\ldots+x_k)-xx_2 \in x_2.$
By the same process, we also have
$
x(x_1+x_2+\ldots+m_{n-1}+x_k) -x(x_1+x_2+\ldots+m_{n-1})-xx_k \in Ru_k.$

It follows that,
$
x(x_1+\ldots+x_k) - xx_1- xx_2-\ldots-xx_k \in \bigcap_{j=1}^{k} Ru_j \subseteq Ru_1 \cap \sum_{j=2}^kRu_j= \lbrace 0 \rbrace.$
Thus $
xm= x\big ( \sum_{i=1}^k x_i \big ) = \sum_{i=1}^k( xx_i).$ Thus $x (\bigoplus_{i=1}^kRu_i)=x (\sum_{i=1}^kRu_i)=\sum_{i=1}^k( \sum_{j=1}^kr_j)Ru_i$ and since for all $j\ne j$ such that $r_jRu_i=0$ then $x (\sum_{i=1}^kRu_i)=\sum_{i=1}^k\sum_{j=1}^k(r_jRu_i)=0$. Thus $x \in \Ann(\bigoplus_{i=1}^k Ru_i ).$ 

For the other inclusion, let $ x \in \Ann(\bigoplus_{i=1}^k Ru_i ).$ We have $x\bigoplus_{i=1}^k Ru_i=0.$ Since we know that $xRu_i=0$. So $x \in \Ann(Ru_i)$ for all $i$. Thus $x \in \bigoplus_{i=1}^k\Ann(Ru_i).$

     \end{proof}
		\section{Concluding comments}
		
		In this article, for $v=0,2,3,c$ we generalize the concept of $v$-prime $R$-module to the notion of $v$-classical prime $R$-module. We proved some few properties about the characterization of classical prime near-ring modules. We recommend as future  work to investigate the radical of classical prime near-ring modules.
	
		\paragraph{Acknowledgments}
			This work was supported by the Council-funded from the Office of Research Development funding at Nelson Mandela University.

	\end{document}